\documentclass[a4paper,11pt,reqno]{amsart}
\usepackage[T1]{fontenc}
\usepackage[utf8]{inputenc}
\usepackage{lmodern}
\usepackage{amsmath, amsthm, amssymb,amscd, mathrsfs, amsfonts, mathtools}
\usepackage{amsmath}
\usepackage{amssymb}
\usepackage{pdfpages}
\usepackage{array}
\usepackage{makecell}
\usepackage{float}
\usepackage{tikz}
\usetikzlibrary{arrows}
\usepackage{multirow}
\usepackage{graphicx}

\usepackage[utf8]{inputenc}
\usepackage{graphicx}

\usepackage[colorlinks]{hyperref}
\usepackage{tikz}

\usepackage{comment}
\usepackage{multirow}
\usepackage{euler}
\usepackage{times}
\usepackage[all]{xy}
\usepackage{todonotes}
\usepackage{xcolor}
\newtheorem{theorem}{Theorem}[section]

\newtheorem{question}[theorem]{Question}
\newtheorem{proposition}[theorem]{Proposition}
\usepackage[lite]{amsrefs}

\renewcommand{\PrintDOI}[1]{\href{http://dx.doi.org/\detokenize{#1}}{doi: \detokenize{#1}}%
  \IfEmptyBibField{pages}{, (to appear in print)}{}}

\def\commutatif{\ar@{}[rd]|{\circlearrowleft}}

\newcommand{\eq}[1][r]
   {\ar@<-3pt>@{-}[#1]
    \ar@<-1pt>@{}[#1]|<{}="gauche"
    \ar@<+0pt>@{}[#1]|-{}="milieu"
    \ar@<+1pt>@{}[#1]|>{}="droite"
    \ar@/^2pt/@{-}"gauche";"milieu"
    \ar@/_2pt/@{-}"milieu";"droite"}

\def\dar[#1]{\ar@<2pt>[#1]\ar@<-2pt>[#1]}
  \entrymodifiers={!!<0pt,0.7ex>+} %%%%%%%%%%%%%%%%%

\newcommand{\bigon}[4][r]{% %%%%% bigons
    \ar@/^1pc/[#1]^{#2}_*=<0.3pt>{}="HAUT"
    \ar@/_1pc/[#1]_{#3}^*=<0.3pt>{}="BAS"
    \ar@{=>} "HAUT";"BAS" ^{#4}
  }

\newcommand{\bigons}[6][r]{  %%%%% Vertical composition of bigons
    \ar@/^2pc/[#1]^{#2}_*=<0.3pt>{}="HAUT"
    \ar@{}    [#1]     ^*=<0.3pt>{}="MILIEUHAUT"
                       _*=<0.3pt>{}="MILIEUBAS"
    \ar[#1]_(0.3){#3}
    \ar@/_2pc/[#1]_{#4}^*=<0.3pt>{}="BAS"
    \ar@{=>} "HAUT";"MILIEUHAUT" ^{#5}
    \ar@{=>} "MILIEUBAS";"BAS" ^{#6}
  }

\newcommand\rTo{\longrightarrow}

\newcommand\mto{\longmapsto}

\newtheorem{thm}{Theorem}[section]

\newtheorem{lem}[thm]{Lemma} 

\newtheorem{cor}[thm]{Corollary}
\theoremstyle{definition}

\newtheorem{corollary}[theorem]{Corollary}

\theoremstyle{definition}
\newtheorem{definition}[theorem]{Definition}
\newtheorem{example}[theorem]{Example}
\theoremstyle{remark}
\newtheorem{remark}[theorem]{Remark}
\numberwithin{equation}{section}

\title{Topologies, Posets and Finite Quandles}

\author{Mohamed Elhamdadi}
\address{Department of Mathematics and Statistics,
University of South Florida, Tampa, FL 33620, U.S.A.}
\email{emohamed@math.usf.edu}

\author{Tushar Gona}
\address{Department of Mathematics, 
University of California,
Berkeley, CA 94720}
\email{gonatushar@berkeley.edu}

\author{Hitakshi Lahrani}
\address{Department of Mathematics and Statistics,
	University of South Florida, Tampa, FL 33620, U.S.A.}
\email{lahrani@usf.edu}

\begin{document}
\maketitle

\begin{abstract} An Alexandroff space is a topological space in which every intersection of open sets is open.  There is one to one correspondence between Alexandroff $T_0$-spaces and \emph{partially ordered sets} (posets).  We investigate Alexandroff $T_0$-topologies on finite quandles.  We prove that there is a non-trivial topology on a finite quandle making right multiplications continuous functions if and only if the quandle has more than one orbit.   Furthermore, we show that right continuous posets on quandles with $n$ orbits are $n$-partite. We also find, for the even dihedral quandles, the number of all possible topologies making the right multiplications continuous. 
Some explicit computations for quandles of cardinality up to \emph{five} are given.
\end{abstract}

\section{Introduction}
Quandles are algebraic structures modeled on the \emph{three} Reidemeister moves in classical knot theory. 
 They have been used extensively to construct invariants of knots and links, see for example \cites{EN, Joyce, Matveev}.  A Topological quandle is a quandle with a topology such that the quandle binary operation is compatible with the topology.  Precisely, the binary operation is continous and the right multiplications are \emph{homeomorphisms}.  Topological quandles were introduced in \cite{Rubin} where it was shown that the set of homomorphisms from the fundamental quandle of the knot to a topological quandle (called also the set of colorings) is an invariant of the knot.  Equipped with the compact-open topology, the set of colorings is a topological space.  In \cite{EM} a foundational account about topological quandles was given.  More precisely,  the notions of ideals, kernels, units, and inner automorphism group in the context of topological quandle were introduced.  Furthermore, modules and quandle group bundles over topological quandles were introduced with the purpose of studying central extensions of topological quandles.  Continuous cohomology of topological quandles was introduced in \cite{ESZ} and compared to the algebraic theories.  Extensions of topological quandles were studied with respect to continuous 2-cocycles, and used to show differences in second cohomology groups for some  specific topological quandles.  Nontriviality of continuous cohomology groups for some examples of topological quandles was shown.  In in \cite{CES} the problem of classification of topological Alexander quandle structures, up to isomorphism, on the real line and the unit circle was investigated.   In \cite{Gr} the author investigated quandle objects internal to groups and topological spaces, extending the well-known classification of quandles internal to abelian groups \cite{Szymik}.  In \cite{Tak} quandle modules over quandles endowed with geometric structures were studied.  The author also gave an infinitesimal description of certain modules in the case when the quandle is a regular s-manifold (smooth quandle with certain properties).  Since any finite $T_1$-space is discrete, the category of finite $T_0$-spaces was considered in \cite{Stong}, where the point set topological properties of finite spaces were investigated.  The homeomorphism classification of finite spaces was investigated and some representations of these spaces as certain classes of matrices was obtained.
 
 This article arose from a desire to better understand the analogy of the work given in \cite{Stong} in the context of \emph{finite topological} quandles.  It turned out that: there is no $T_0$-topology on any finite  connected (meaning one orbit under the action of the Inner group) quandle  $X$ that makes $X$ into a topological quandle (Theorem~\ref{noT0}).  Thus we were lead to consider topologies on quandles with more than \emph{one} orbit.  It is well known \cite{Alex} that the category of Alexandroff $T_0$-spaces is equivalent to the category of \emph{partially ordered sets} (posets).  In our context, we prove that for a quandle $X$ with more than one orbit, there exists a unique non trivial topology which makes right multiplications of $X$ continuous maps (Proposition~\ref{Prop}).  Furthermore, we prove that if $X$ be a finite quandle with two orbits $X_1$ and $X_2$ then any continuous poset on $X$ is biparatite with vertex set $X_1$ and $X_2$ (Proposition~\ref{Bipar}).  This article is organized as follows.  In Section~\ref{Review} we review the basics of topological quandles.  Section~\ref{Poset} reviews some basics of posets, graphs and some hierarchy of separation axioms. In Section~\ref{Main} the main results of the article are given. 
 Section~\ref{Computations} gives some explicit computations based on some computer softwares (Maple and Python) of quandles up to order \emph{five}.  

\section{Review of Quandles and Topological Quandles}\label{Review}
A quandle is a set $X$ with a binary operation $*$ satisfying the following three axioms:
\begin{enumerate}
    \item 
    For all $x$ in $X$, $x*x=x,$
    \item
    For all $y,z \in X$, there exists a unique $x$ such that $x*y=z$, 

    \item
    For all $x,y,z \in X$, $(x*y)*z=(x*z)*(y*z)$.
    
\end{enumerate}
These three conditions come from the axiomatization of the three Reidemeister moves on knot diagrams.  The typical examples of quandles are:
(i) Any Group $G$ with conjugation $x*y=y^{-1}xy$, is a quandle called the \emph{conjugation quandle} and (ii) Any group $G$ with operation given by $x*y=yx^{-1}y$, is a quandle called the \emph{core quandle}.\\
Let $X$ be a quandle.  For an element $y \in X$, left multiplication $L_y$ and right multiplication $R_y$ by an element $y$ are the maps from $X$ to $X$ given respectively by $L_y(x):=y*x$ and $R_y(x)=x*y$.  A function $f: (X,*) \rightarrow  (X,*)$ is a quandle {\em homomorphism} if for all $x,y \in X, f(x * y)=f(x) * f(y)$.  If furthermore $f$ is a bijection then it it is called an \emph{automorphism} of the quandke $X$.  We will denote by {\rm Aut(X)} the automorphism group of $X$.   The subgroup of {\rm Aut(X)}, generated by the automorphisms $R_x$, is called the {\em inner} automorphism group of $X$ and denoted by {\rm Inn}$(X)$.  If the group {\rm Inn}$(X)$ acts \emph{transitively} on $X$, we then say that $X$ is connected quandle meaning it has only one orbit.  Since we do not consider topological connectedness in this article, then through the whole article, the word connected quandle will stand for algebraic connectedness.  For more on quandles refer to \cite{EN, Joyce, Matveev, E}.
Topological quandles have been investigated in \cite{CES, EM, Rubin, ESZ}.  Here we review some basics of topological quandles.
\begin{definition}
A \textit{topological quandle} is a quandle $X$ whith a topology such that the map $X\times X\ni (x,y)\mto x * y\in X$ is a continuous, the right multiplication $R_x:X\ni y\mto y* x\in X$ is a homeomorphism, for all $x\in X$, and $x* x=x$. 
\end{definition}
It is clear that any finite quandle is automatically a topological quandle with respect to the discrete topology.

\begin{example} \cite{CES}
 	Let $(G, +)$ be a topological abelian group and let $\sigma$ be a continuous automorphism of $G$.  The continuous binary operation on $G$ given by $x*y=\sigma(x)+(Id-\sigma)(y), \forall x,y \in G,$ makes $(G,*)$ a topological quandle called \textit{topological Alexander quandle}. In particular, if $G=\mathbb{R}$ and $\sigma(x)=tx$ for non-zero $t \in \mathbb{R}$, we have the topological Alexander structure on $\mathbb{R}$ given by $x*y=tx+(1-t)y$.    
\end{example}

\begin{example} The following examples were given in \cites{Rubin,EM}.  The unit sphere $\mathbb{S}^n \subset \mathbb{R}^{n+1}$ with the binary operation $x\ast y=2(x\cdot y)y-x$ is a topological quandle, where $\cdot$ denotes the inner product of $\mathbb{R}^{n+1}$. Now consider $\lambda$ and $\mu$ be real numbers, and let $x,y \in S^n$. Then $$\lambda x * \mu y=\lambda[2{\mu}^2(x\cdot y)y-x].$$ 
  In particular, the operation $$\pm x *  \pm y=\pm (x * y)$$ provides a structure of topological quandle on the quotient space that is the projective space $\mathbb{RP}^n$.
\end{example}

\section{Review of topologies on finite sets, Posets and Graphs}\label{Poset}
Now we review some basics of directed graphs, posets and $T_0$ and $T_1$ topologies. %some hierarchy of separation axioms. 

\begin{definition}
A \emph{directed} graph G is a pair $(V, E)$ where $V$ is the set of vertices and $E$ is a list of
 directed line segments called edges between pairs of vertices.  
\end{definition}
An edge from a vertex $x$ to a vertex $y$ will be denoted symbolically by $x < y$ and we will say that $x$ and $y$ are \emph{adjacent}. 
The following is an example of a directed graph.
\begin{example}
Let $G=(V,E) $ where $V=\{a,b,c,d\}$ and $E=\{b<a,c<a ,a<d \}$.
\end{example}
\begin{center}
\begin{figure}[h!]

\begin{tikzpicture}
[nodes={draw, circle}, <-]
\node{d}
    child { node {a} 
        child { node {b} }
        child { node {c} }  
    }
    child [ missing ];
    
\end{tikzpicture}
\end{figure} 
\end{center}

\begin{definition}
An \emph{idependent set} in a graph is a set of pairwise non-adjacent vertices.
\end{definition}
\begin{definition}
A (directed) graph $G=(V,E)$ is called biparatite if  $V$ is the union of two disjoint independent sets $V_1$ and $V_2$.
%can be split into two disjoint subsets, $V_1$ and $V_2$, 
%in such a way that each edge in the graph joins a vertex in $V_1$ to a vertex in $V_2$.
\end{definition}
\begin{definition}
A (directed) graph $G$ is called \emph{complete  biparatite} if $G$ is bipartite  and for every $v_1 \in V_1$  and $v_2 \in V_2$ there is an edges in $G$ that joins $v_1$ and $v_2$.% then the graph $G$ is called a complete  bipartite graph.
\end{definition}
\begin{example} 
 Let $V=V_1 \cup V_2$ where $V_1=\{4,5\}$ and $V_2=\{1,2,3\}$.  Then the directed graph $G=(V,E)$ is complete biparatite graph.
    \begin{center}

\includegraphics[]{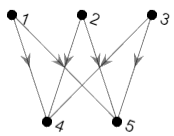}

\end{center}
\end{example}
Now we recall the definition of partially ordered set.
\begin{definition}
A partially ordered set (poset) is a set $X$ with an order denoted $\leq$ that is reflexive, antisymmetric and transitive.
\end{definition}

\begin{example}
For any set $X$, the power set of $X$ ordered by the set inclusion relation $\subseteq$ forms a poset $( \mathcal{P}(X),\subseteq)$
\end{example}
\begin{definition}
Two partially ordered sets $P=(X,\leq)$ and $Q=(X,\leq')$  are said to be isomorphic  if there exist a bijection $f:X \rightarrow X'$ such that $x \leq y$ if and only if   $f(x)\leq'f(y).$
\end{definition}

\begin{definition}\label{ConnPoset}
A poset $(X,\leq )$ is connected if for all $x,y\in X$, there exists sequence of elements $x=x_1,x_2, \ldots, x_n=y $ such that every two consecutive elements $x_i$ and $x_{i+1}$ are comparable (meaning $x_i < x_{i+1}$ or $x_{i+1} < x_i$).
\end{definition}
\noindent
{\bf Notation:} Given an order $\leq$ on a set $X$, we will denote $x<y$ whenever $x \neq y$ and $x \leq y$.  Finite posets $(X,\leq)$ can be drawn as directed graphs where the vertex set is $X$ and an arrow goes from $x$ to $y$ whenever $x \leq y$.  %For the purpose of not encumbering the directed graph, 
For simplicity, we will not draw loops which correspond to $x \leq x$.
We will then use the notation $(X,<)$ instead of $(X,\leq)$ whenever we want to ignore the reflexivity of the partial order.%, particularly in drawing diagrams of posets. 

\begin{example}. Let $X=\mathbb{Z}_8$ be the set of integers modulo $8$.  The map $f:X \rightarrow X$ given by $f(x)=3x-2$ induces an isomorphism between the following two posets $(X,<)$ and $(X,<')$.

  \begin{center}
      
   \includegraphics{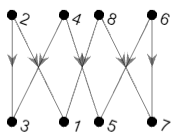},  \qquad    
     \includegraphics{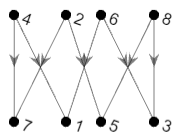} 
   \end{center}
\end{example}
\begin{definition}
A chain in a poset $(X, <)$ is a subset $C$ of $X$ such that the restriction of $<$ to $C$ is a total order (i.e. every two elements are comparable).
\end{definition}

Now we recall some basics about topological spaces called $T_0$ and $T_1$ spaces.
\begin{definition}
	A topological space $X$ is said to have the property $T_0$ if for every pair of distinct points of $X$, at least one of them has a neighborhood not containing the other point.
\end{definition}

\begin{definition}
	A topological space $X$ is said to have the property $T_1$ if for every pair of distinct points of $X$, each point has a neighborhood not containing the other point.
\end{definition}
Obviously the property $T_1$ implies the property $T_0$.  Notice also that this definition is equivalent to saying singletons are closed in $X$.  Thus a $T_1$-topology on a \emph{finite} set is a discrete topology.

%The number of $T_0$-inequivalent topologies on sets of cardinalities $2,3,4,5, 6, 7$ and $8$ are respectively $2, 5, 16, 63, 318,2045, $ and $16999$.
Since any finite $T_1$-space is discrete, we will focus on the category of finite $T_0$-spaces.  First we need some notations. 

Let $X$ be a finite topological space.  For any $x \in X$, we denote 
\[
U_x:=\textit{the smallest open subset of $X$ containing $x$}
\]

It is well known \cite{Alex} that the category of $T_0$-spaces is isomorphic to the category of posets.  We have $x \leq y$ if and only if $U_y \subseteq U_x$ which is equivalent to $C_x \subset C_y$, where $C_v$ is the complelement $U_v^{c}$ of $U_v$ in $X$.  Thus one obtain that $U_x=\{w \in X;\; x \leq w\}$ and $C_x=\{v \in X; \;v < x\}$.  Under this correspondence of categories, the subcategory of finite posets
is equivalent to the category of finite $T_0$-spaces. \\
Through the rest of this article we will use the notation of $x<y$ in the poset whenever $x\neq y$ and $x \leq y$.

%The following table gives $T_0$-inequivalent topologies on finite sets of cardinality $n$, where $1 \leq n \leq 6$.

%$$\begin{array}{|c|c| c|} 
%\hline
%n \ &     T_0-topologies   \\ \hline
%\ 1 &  1  \\ \hline
%\ 2 &   2   \\ \hline
%\ 3 &   5   \\ \hline
%\ 4 &   16   \\ \hline
%\ 5 &   63  \\ \hline
%\ 6 &  	  318  \\ \hline
%\end{array}
%

\begin{comment}
$$\begin{array}{|c|c| c| c| c| c|} 
\hline
n \ & Distinct\\& Topologies & Distinct\\&& T_0-Topologies & inequivalent\\&&& topologies & T_0-inequivalent\\&&&& topologies   \\ \hline
\ 1 & 1 &1 & 1 & 1  \\ \hline
\ 2 &  4 & 3 & 3 & 2   \\ \hline
\ 3 &  29 & 19 & 9 & 5   \\ \hline
\ 4 &  355 & 219 & 33 & 16   \\ \hline
\ 5 &  6942 & 4231 & 139 & 63  \\ \hline
\ 6 &  209527 & 130023 & 718	 & 318  \\ \hline
\end{array}
$$
\end{comment}

\section{Topologies on non- connected Quandles}\label{Main}

As we mentioned earlier, since $T_1$-topologies on a finite set are discrete, we will focus in this article on $T_0$-topologies on \emph{finite quandles}.  A map on finite spaces is continuous if and only if it preserves the order.  It turned out that on a finite quandle with a $T_0$-topology, left multiplications can not be continuous as can be seen in the following theorem

\begin{theorem}\label{left}
Let $X$ be a finite quandle endowed with a $T_0$-topology.  Assume that for all $z \in X$, the map $L_z$ is continuous, then $x \leq y$ implies $L_z(x)=L_z(y)$.

\end{theorem}
\begin{proof} We prove this theorem by contradiction.  Let $X$ be a finite quandle endowed with a $T_0$-topology.  Assume that $x \leq y$ and $L_z(x) \neq L_z(y)$.  If $x=y$, then obviously $L_z(x)=L_z(y)$.  Now assume $x <y$, then for all $a \in X$, the continuity of $L_a$ implies that $a*x \leq a*y$.  Assume that there exist $a_1 \in X$ such that, $z_1:=a_1*x=L_{a_1}(x) < a_1*y=L_{a_1}(y)$.
The invertibility of right multiplications in a quandle implies that there exist unique $a_2$ such that 
$a_2*x=a_1*y$ hence 
$a_1*x<a_2*x$ which implies $a_1 \neq a_2.$  Now we have $a_1*x<a_2*x \leq a_2 *y=z_2$. 
We claim that $a_2*x < a_2 *y$. 
if $a_2*y=a_2*x$ and since $a_2*x=a_1*y$ we will have
$a_2*y=a_2*x=a_1*y$ hence $a_2*y=a_1*y$ but $a_1 \neq a_2$, thus contradiction.
Now that we have proved $a_2*x < a_2 *y$, then there exists $a_3$ such that $a_2*y=a_3*x$  we get, $a_2*x<a_3*x$ repeating the above argument we get, $a_3*x<a_3*y$.  Notice that $a_1,a_2$ and $a_3$ are all pairwise disjoint elements of $X$.
Similarly, we construct an \emph{infinite} chain, $a_1 * x< a_2*x <a_3 *x < \cdots $, which is impossible since $X$ is a finite quandle.  Thus we obtain a contradiction.
  \end{proof}
 % {\color{red} Stoped here on Thursday sep 22}.\\
\begin{comment}
Theorem~\ref{left} states that we can not have a notion of \emph{semitopological quandles} (continuity with respect to each variable considered separately) on finite quandles.  For more on semitopological groups see \cite{Husain}.  Thus through the rest of the paper, the word continuity will mean continuity of \emph{right multiplications} in the finite quandle.  We will then say that the quandle with its topology is \emph{right continuous}.  We borrow the following terminologies from the book \cite{Rup}.
\end{comment}
We have the following Corollary
\begin{corollary}
Let $X$ be a finite quandle endowed with a $T_0$-topology.
 If $C$ is a chain of $X$ as a poset then any left continuous function $L_x$ on $X$ is a constant function on $C$.
\end{corollary}

\begin{definition}
    A quandle with a topology in which right multiplications (respectively left multiplications) are continuous is called \emph{right topological quandle} (respectively \emph{left topological quandle}).
\end{definition}
%Since we start with a quandle $X$ before adding a topology to it, then
In other words, right topological quandle means that for all $x,y,z \in X$, 
\[
x<y \implies x*z < y*z.
\]
%But since left multiplications are not necessarly bijective maps, 

and, since left multiplications are not necessarly bijective maps, left topological quandle means that for all $x,y,z \in X$, 
\[
x<y \implies z*x \leq z*y.
\]

\begin{comment}
\begin{definition}
    A quandle with a topology in which both right multiplications and left multiplications are continuous is called \emph{semitopological quandle}.
\end{definition}

Obviously topological quandle implies semitopological quandle which implies 
right topological quandles, but the converses are not true.  
\end{comment}

\begin{theorem}\label{noT0}
There is no $T_0$-topology on a finite  connected quandle $X$ that makes $X$ into a right topological quandle. 
\end{theorem}
\begin{proof}
   Let $x <y$.  Since $X$ is connected quandle, there exists $\phi \in Inn(X)$ such that $y=\phi(x)$.  Since $X$ is finite, $\phi$ has a finite order $m$ in the group $Inn(X)$.  Since $\phi$ is a continous automorphism then $x<\phi(x)$ implies $x<\phi^m(x)$ giving a contradiction.
\end{proof}
\begin{corollary}\label{noodd}
There is no $T_0$-topology on any latin quandle that makes it into a right topological quandle. 
\end{corollary}
\begin{comment}
\begin{proof}
For Dihedral quandle $R_y(x)=2y-x$\\
If n is odd then 2 is invertible in $Z_n$, hence for every\\
$z \in \mathbb(z)_n$ there exist $y=\frac{z+x}{2}$ such that $R_y(x)=z$
hence $R_y(x)$ is both injective in $x$ and $y$ hence quandle will have only one orbit.\\
Hence it is connected.
\end{proof}
\end{comment}

Thus Theorem~\ref{noT0} leads us to consider quandles $X$ that are not  connected, that is $X=X_1 \cup X_2\cup \ldots X_k$ as orbit decomposition, search for $T_0$-topology on $X$ and investigate the continuity of the binary operation.  

%For simplicity assume $X=X_1 \cup X_2$ as orbit decomposition.  Let $x<y$ then $x$ and $y$ belong to disjoint orbits since orbits are connected.  Thus assume $x \in X_1$ and $y \in X_2$.

%The following proposition appeared in \cite{BPS} but we state it in slightly different form for our purpose. 
\begin{proposition}\label{Prop}
Let $X$ be a finite quandle with orbit decomposition $X=X_1 \cup \{a\}$, then there exist unique non trivial $T_0$-topology which makes $X$ right continuous. 
\end{proposition} 

\begin{proof}
Let $X=X_1 \cup \{a\}$ be the orbit decomposition of the quandle $X$.  For any $x,y \in X_1$, there exits $\phi \in Inn(X)$ such that $\phi(x)=y$ and $\phi(a)=a$.  Declare that $x<a$, then $\phi(x)<a$.  Thus for any $z\in X_1$ we have $z<a$. Uniqueness is obvious.  
\end{proof}

%\begin{proof}
 % Let $X=X_1 \cup X_2$ where $X_1=\{1,2...n-1 \}$ and $X_2 =\{n\}$.
 % For every $i,j \in X_1$, there exist $\phi \in Inn(X)$ such that $\phi(i)=j$ and $\phi(n)=n$.
 %finite series $k_1, k_2...k_t$ such that \\
 %$((((i*k_1)*k_2)*...k_t)=j$\\
 %If $1<n$ then $\phi(1)<n$, giving $i <n$ for each  $1\leq i \leq n-1$.% which forms unique poset. 
%\end{proof}
The $T_0$-topology in Proposition~\ref{Prop} is precisely given by $x<a$ for all $x \in X_1.$
%\end{proposition}
%Since the odd dihedral quandle is connected, then by corollary \ref{noodd} we will consider dihedral quandle of even order.

%\newpage

\begin{proposition} \label{Bipar}
 Let $X$ be a finite quandle with two orbits $X_1$ and $X_2$.  Then any right continuous poset on $X$ is biparatite with vertex set $X_1$ and $X_2$. 
\end{proposition}

\begin{proof}
We prove this proposition by contradiction.  For every $x_1, y_1 \in X_1$ such that $x_1 <y_1$. We know that there exist $\phi \in Inn(X)$ such that $\phi(x_1)=y_1$. Hence, $x_1< \phi(x_1)$ implies $x_1 <\phi^m(x_1)=x_1$, where $m$ is the order of $\phi$ in $Inn(X)$.  Thus we have a contradiction.

\end{proof}
\begin{proposition}\label{biComp}
Let $X$ be a finite quandle with two orbits $X_1$ and $X_2$.  Then the complete bipartite graph with vertex set $X_1$ and $X_2$ forms a right continuous poset.
\end{proposition}

\begin{proof} Let $X$ be a finite quandle with two orbits $X_1$ and $X_2$.
If $x \in X_1$ and $y \in X_2$ then for every $\phi \in Inn(X)$ we have  $\phi(x) \in X_1$ and $\phi(y) \in X_2$.  Proposition~\ref{Bipar} gives that the graph is bipartite and thus $x<y$.  We then obtain $\phi(x) < \phi(y)$ giving the result. % Thus the poset operation $<$ is invariant on the right by the quandle multiplication.%a right continuous poset.
\end{proof}

\begin{remark}
By Proposition~\ref{biComp} and Theorem~\ref{left}, there is a non-trivial $T_0$-topology making $X$ right continuous if and only if the quandle has more than one orbit. 
\end{remark}

Notice that Proposition~\ref{biComp} can be generalized to $n$-paratite  complete graph.

%In generality, if  $(X,*) $ be a finite  quandle with $n$ orbits $X_1$, $X_2, \ldots, X_n$.  % on $Z_{N}$    Then $n$-paratite  complete graph with vertex sets $X_1$,$X_2, \dots, X_n$ forms a continuous poset.

%\newpage

%\includegraphics[]{dihedral6.png}

The following table gives the list of right continuous posets on some even dihedral quandles.  In the table, the notation $(a,b)$ on the right column means $a<b.$
\begin{table}[H]
\caption{Right continuous posets on dihedral quandles}\label{TTable1}
%Note: Dihedral Quandle when n is odd has unique orbit which makes it connected quandle and by proposition it has only trivial poset.\\
%Lets study Dihedral quandle where n is even.\\
%\newpage
\begin{center}
\begin{tabular}{|c|c|} 
    \hline
        Quandle  & 
        Posets \\
        
    \hline
        $R_4$   & 
        ((0,1),(2,1),(0,3),(2,3)) \\
        \hline
        $R_6$ & ((0, 1), (0, 5), (2, 1), (2, 3), (4, 3), (4, 5))\;; \\
  & ((0,3), (2, 5), (4, 1)). \\
  
 \hline
 $R_8$ & ((2, 7), (4, 7), (6, 1), (6, 3), (0, 5), (2, 5), (4, 1), (0, 3))\;;\\
  & (( 0, 1), (6, 7), (4, 5), (0, 7), (2, 1), (2, 3), (4, 3), (6, 5)).  \\ \hline
  $R_{10}$ & ((0, 1), (6, 7), (4, 5), (2, 1), (8, 9), (2, 3), (4, 3), (8, 7), (0, 9), (6, 5)) \;;\\

 &  ((4, 7), (6, 9), (2, 9), (8, 1), (8, 5), (0, 7), (6, 3), (2, 5), (4, 1), (0, 3))\;;\\

&  ((2, 7), (8, 3), (0, 5), (4, 9), (6, 1)).\\
\hline
 \end{tabular}       
 \end{center}
  \end{table}    
 
Notice that in table~\ref{TTable1}, the dihedral quandle $R_4$ has only one right continuous poset $((0,1),(2,1),(0,3),(2,3))$ which is complete biparatite.  While the dihedral quandle $R_6$ has two  continuous posets $((0, 1), (0, 5), (2, 1), (2, 3), (4, 3), (4, 5))$ and $((0,3), (2, 5), (4, 1))$ illustrated below.

     \begin{center}   
  \begin{tikzpicture}
    [scale=.8,auto=left,every node/.style={circle,fill=gray}]
  
  \node (n6) at (2,10) {3};
  \node (n4) at (2,7)  {0};
  \node (n5) at (3,10)  {5};
  \node (n1) at (3,7) {2};
  \node (n2) at (4,10)  {1};
  \node (n3) at (4,7)  {4};

  \foreach \from/\to in {n6/n4,n5/n1,n2/n3}
    \draw (\from) -> (\to);

\end{tikzpicture}
\qquad \qquad \qquad
\begin{tikzpicture}
    [scale=.8,auto=left,every node/.style={circle,fill=gray}]
  
  \node (n6) at (14,10) {3};
  \node (n1) at (13,7)  {2};
  \node (n5) at (13,10)  {5};
  \node (n4) at (14,7) {0};
  \node (n2) at (15,10)  {1};
  \node (n3) at (15,7)  {4};

  \foreach \from/\to in {n4/n2,n2/n1, n5/n3,n6/n1,n5/n3,n4/n5, n3/n6}
    \draw (\from) -- (\to);

\end{tikzpicture}
\end{center}
Moreover, in table~\ref{TTable1}, for $R_8$ the bijection $f$ given by  $f(k)=3k-2$ makes the two posets isomorphic.  The same bijection gives isomorphism between the first two posets of $R_{10}$.
%One can check that this poset is right continous on $Q$ and is not complete bipartite.
 The following Theorem characterizes non complete biparatite posets on dihedral quandles.
 \begin{theorem}
  Let $R_{2n}$ be a dihedral quandle of even order.  Then $R_{2n}$ has $s+1$ right continuous posets, where $s$ is number of odd natural numbers less than n and relatively non coprime with $n$
 \end{theorem}
\begin{proof}
Let $X=R_{2n}$ be the dihedral quandle with orbits $X_1=\{0,2, \ldots, 2n-2\}$ and $X_2=\{1,3, \ldots ,2n-1\}$.  For every $x \in X_2$, we construct a partial order $<_x$ on $R_{2n}$, such that for all $y \in X$, we have $2y<_x2y-x$ and  $2y <_x 2y+x$.
Then $<_x$ is clearly right continuous partial order since $2y<2y-x$ and $2y<2y+x$ for all $y$ imply that $2z-2y<2z-(2y-x)$.  In other words we obtain   $2y*z<(2y-x)*z$. From the definition of the order $<_x$ it is clear the two partial orders $<_x$  and $<_{2n-x} $ are the same.
  Hence we obtain the following distinct partial orders $<_1, <_3, \ldots  $.  Now we check which ones are isomorphic.  If $m$ is odd and  $gcd(n,m)=1$ then $f(k)=mk-2$ is a bijective function  making $<_1$ and $<_m$ isomorphic.
 %But when  $m$ is odd and  $gcd(m,n)=k>1$ gives $gcd(2m,2n)=2k$.  
 Now let $m$ be odd and  $gcd(m,n)=k>1$.  The two posets $<_1$ and $<_m$ are non isomorphic since $<_1$ is connected poset, as in Definition~\ref{ConnPoset}, and $<_m$ is not connected poset.
 %$<_1$ has one connected component as a poset:since $0<_1 1, 2<_11, 2<_13, 4<_13, $ so on we will have $2(n-1)<_1 1 $  and for partial order  $<_m$   $0<m, 2m<m, 2m<3m,4m<3m,4m<5m,6m<5m$ and so on we will have $2kml=2n=0 $ will be in one connected component and $2, 2+m, 2+2m,.. $form another connected component hence it will divide $X_1$ in $m$ components. \\
 We show that these are the only right continuous posets.  Given a right continuous poset on $R_{2n}$ then $ a < b$ can be written
as $a < a-(a-b)$ which implies that  $a <_x b$ where $x=a-b.$ Now if $a < b$ then by Proposition~\ref{Bipar}, we have $a \in X_1$, $b \in X_2$.  Now let $a = 2\alpha$ and $b = 2\beta+1$ then $a-b = 2(\alpha-\beta)-1 \in X_2$.
%Furthermore, $a * z = 2z - a$ and $b * z = 2z - b$ then $a * z <_x b* z$.  
This ends the proof.

\end{proof} 

 \begin{corollary}

 For the dihedral quandle $R_{2^n}$ with $2^n$ elements,  there is a unique right continuous poset.
 \end{corollary}
% As seen above there are two right continuous posets on $R_8$ which are isomorphic to each other hence there is only one right continuous poset upto isomorphisim.
 %\vspace*{-65pt}
 
 \vspace*{-7pt}
 \section{Some Computer Calculations}\label{Computations}
 
 In this section we give non-trivial right and left  continuous posets on the finite quandles of order up to $5$ based on Maple and Python computations.  In the following tables we have excluded the trivial and connected quandles.% and also corresponding trivial posets.
\begin{table}[H]
\caption{Continuous posets on quandles of order $3$}
\label{Table2}
\begin{center}
\begin{tabular}{ |c|c|c|} 
    \hline
       Quandle for n = 3 & 
       Right continuous Posets & 
        Left continuous poset \\
    \hline
        \small{
        $\left[ \begin{array}{c}
        0 \;0 \;1  \\
        1 \;1 \;0 \\
        2 \;2\; 2  
        \end{array} \right] $} 
            & 
                ((0,2),(1,2)) 

            & ((0,1))\\ 
     
    \hline
\end{tabular}
\end{center}
\end{table}

 As seen in Table~\ref{Table2} for $n=3$, there exist a unique right continuous poset and a unique left continuous poset.

\begin{table}[H]
\caption{Continuous posets on quandles of order 4}
\label{Table3}
\begin{center}
\begin{tabular}{ |c|c|c|} 
    \hline
        Quandles for n = 4 & 
        Right continuous poset &
        Left continuous poset \\ 
    \hline
        \small{
        $\left[ \begin{array}{c}
        0 \;0 \;0 \;0  \\
        1 \;1 \;1 \;2 \\
        2 \;2 \;2 \;1 \\
        3 \;3 \;3 \;3
        \end{array} \right] $} 
           & \makecell{
                ((0, 3))\;;\\
                ((0, 1), (0, 2), (0, 3))\;;\\
                ((0, 1), (0, 3), (1, 2))\;;\\
                ((0, 1), (0, 2), (1, 3), (2, 3))\;;\\
                 ((2, 3), (1, 3))\;;\\
                 ((2, 3), (1, 3), (0, 3)).\\
                } 
           & ((0,1),(1,2)) and ((1,2)) \\ 
    \hline
        \small{
        $\left[ \begin{array}{c}
         0 \;0 \;0 \;1  \\
        1 \;1\; 1 \;2 \\
        2 \;2 \;2 \;0 \\
         3 \;3 \;3 \;3
        \end{array} \right] $} 
           & \makecell{
                    ((0,3),(1,3),(2,3))\\ 
                            } 
           & ((0,1),(1,2)) and ((1,2)) \\ 
    \hline
        \small{
        $\left[ \begin{array}{c}
        0 \;0 \;1 \;1  \\
        1 \;1 \;0 \;0 \\
        2 \;2 \;2 \;2 \\
        3 \;3 \;3 \;3
        \end{array} \right] $} 
            & \makecell{
                ((0,2),(1,2),(0,3),(1,3),(2,3))\;; \\
                ((0,2),(1,2),(0,3),(1,3)) \;;\\
                ((0,2),(1,2)) \;;\\
                ((2,3)).\\
                } 
            & 
            ((0,1),(2,3)) and ((2,3))\\

    \hline
        \small{
        $\left[ \begin{array}{c}
         0 \;0 \;0 \;0  \\
        1 \;1 \;3 \;2 \\
        2 \;3 \;2 \;1 \\
        3 \;2 \;1 \;3
        \end{array} \right] $} 
            & \makecell{
                ((0,1),(0,2),(0,3)) \\
        } 
            & None \\ 
    \hline
        \small{
        $\left[ \begin{array}{c}
        0 \;0 \;1 \;1  \\
        1 \;1 \;0 \;0 \\
        3 \;3 \;2 \;2 \\
        2 \;2 \;3 \;3
        \end{array} \right] $} 
            & \makecell{
                ((0,2),(0,3),(1,2),(1,3)) \\
            } 
            & ((0,1),(2,3)) \\ 
    
    \hline
\end{tabular}
\end{center}
\end{table}

 \begin{table}[H]
\caption{Continuous posets on quandles of order 5, Part I}
\label{Table4}
\begin{center}
\begin{tabular}{ |c|c|c|} 
    \hline
       Quandles for  n = 5& 
        Right continuous& 
        Left continuous 
 \\
       
    \hline
        
        \small{
        $\left[ \begin{array}{c}
        0\; 0\; 0\; 0\; 0 \\
        1\; 1\; 1\; 1\; 1 \\
        2\; 2\; 2\; 2\; 3 \\ 
        3\; 3\; 3\; 3\; 2 \\
        4\; 4\; 4\; 4\; 4
        \end{array} \right] $} 
            & \makecell{\\
            ((0,1),(1,2),(1,3),(0,4))\;;\\
            ((0,2),(0,3),(1,2),(1,3),(4,2),(4,3))\;;\\
            ((0,2),(0,3),(1,2),(1,3),(2,4),(3,4))\;;\\
            ((0,1),(1,4),(4,2),(4,3)).\\
           \\
           }  &
            \makecell{ ((0,1),(1,2),(2,3))\;;\\
            ((0,1),(1,2))\;;\\
            ((1,2)).}
            \\ 
    \hline
        \small{
        $\left[ \begin{array}{c}
        0\; 0\; 0\; 0\; 0 \\
        1\; 1\; 1\; 1\; 2 \\
        2\; 2\; 2\; 2\; 3 \\ 
        3\; 3\; 3\; 3\; 1 \\
        4\; 4\; 4\; 4\; 4
        \end{array} \right] $} 
            & \makecell{
            ((0,1), (0,2), (0,3), (2,4), (3,4), (1,4))\;; \\
((0,4)) \;;\\
((0,1), (0,2), (0,3))\;; \\
(( 0,4 ),(4,1 ),(4,2 ),(4,3)). 
}	& \makecell{\\((0,1), (1,2), (2,3)) \;;\\
((0,1), (1,2)) \;;\\
((2,3)).
\\
}
  \\ 
    \hline
        \small{
        $\left[ \begin{array}{c}
        0\; 0\; 0\; 0\; 1 \\
        1\; 1\; 1\; 1\; 0 \\
        2\; 2\; 2\; 2\; 3 \\ 
        3\; 3\; 3\; 3\; 2 \\
        4\; 4\; 4\; 4\; 4
        \end{array} \right] $} 
            & \makecell{\\
            ((1,2),(0,3),(2,4),(3,4))\;;\\
((1,2),(0,2),(1,3),(0,3),(2,4),(3,4))\;;\\
((1,4),(0,4))\;;\\
((1,2),(0,2),(1,3),(0,3)).\\
\\}
	& \makecell{ ((1,2),(0,1),(2,3))\;; \\ 
	((0,1),(0,2))\;;
	\\ ((0,2),(1,2)).\\
	\\}\\
	
    \hline
        \small{
        $\left[ \begin{array}{c}
        0\; 0\; 0\; 0\; 1 \\
        1\; 1\; 1\; 1\; 2 \\
        2\; 2\; 2\; 2\; 3 \\ 
        3\; 3\; 3\; 3\; 0 \\
        4\; 4\; 4\; 4\; 4
        \end{array} \right] $} 
            & \makecell{((0,4),(1,4),(2,4),(3,4)).}	& \makecell{\\
            ((1,2),(0,1),(2,3))\;; \\ 
	((0,1),(0,2))\;;
	\\ ((0,2),(1,2)).\\
	\\
\\	}		\\ 
    \hline
        \small{
        $\left[ 
        \begin{array}{c}
        0\; 0\; 0\; 0\; 0 \\
        1\; 1\; 1\; 1\; 1 \\
        2\; 2\; 2\; 4\; 3 \\ 
        3\; 3\; 4\; 3\; 2 \\
        4\; 4\; 3\; 2\; 4
        \end{array} \right] $} &\makecell{\\
        ((0,2),(0,3),(0,4))\;;\\
        ((0,2),(0,3),(0,4),(1,2),(1,3),(1,4))\;;\\
        ((0,1))\;;\\
        ((0,2),(0,3),(0,4),(0,1),(1,2),(1,3),(1,4)).\\
       \\
       \\
       }
            & \makecell{((0,1),(0,2),(0,3))\;;\\
            ((0,1),(0,2))\;;\\
            ((0,1)). }  \\ 
    \hline
        \small{
        $\left[ \begin{array}{c}
        0\; 0\; 0\; 0\; 0 \\
        1\; 1\; 1\; 2\; 2 \\
        2\; 2\; 2\; 1\; 1 \\ 
        3\; 3\; 3\; 3\; 3 \\
        4\; 4\; 4\; 4\; 4
        \end{array} \right] $} 
            & \makecell{\\
            ((1,3),(2,3))\;;\\
            ((1,3),(2,3),(1,4),(2,4))\;;\\
            ((0,4))\;;\\
            ((3,2),(3,1))\;;\\
            ((1,3),(2,3),(4,1),(4,2)).\\
            \\
              }
            & \makecell{((0,1),(1,2),(3,4))\;; \\ 
            ((0,1),(1,2))\;;\\
            ((0,1),(3,4))\;;\\
            ((0,1)).\\}
            \\
            
    \hline
        \small{
        $\left[ \begin{array}{c}
        0\; 0\; 0\; 0\; 0  \\
        1\; 1\; 1\; 2\; 2 \\
        2\; 2\; 2\; 1\; 1 \\ 
        3\; 4\; 4\; 3\; 3 \\
        4\; 3\; 3\; 4\; 4
        \end{array} \right] $} 
            & \makecell{((0,1),(0,2),(0,3),(0,4))\;;\\
            ((0,1),(0,2))\;;\\
            ((0,1),(0,2),(2,3),(2,4),(1,3),(1,4))\;;\\
            ((1,3),(1,4),(2,3),(2,4)).} & \makecell{((0,1),(1,2),(3,4))\;; \\ 
            ((0,1),(1,2))\;;\\
            ((0,1),(3,4))\;;\\
            ((0,1)).\\}  \\ 
    \hline
\end{tabular}
\end{center}
\end{table} 

 \begin{table}[H]
\caption{Continuous posets on quandles of order 5, Part II}
\label{Table1}
\begin{center}
\begin{tabular}{ |c|c|c|} 
    \hline
        Quandles for n = 5 & 
        Right continuous & 
        Left continuous 
                 \\ 
    \hline
        \small{
        $\left[ \begin{array}{c}
        0\; 0\; 0\; 1\; 1 \\
        1\; 1\; 1\; 0\; 0 \\
        2\; 2\; 2\; 2\; 2 \\ 
        3\; 3\; 4\; 3\; 3 \\
        4\; 4\; 3\; 4\; 4
        \end{array} \right] $} 
            & \makecell{
((0,2), (1,2) (2,3),(2,4))\;;\\
((0,2), (1,2))\;;\\
((2,3 ),(2,4))\;;\\
((0,3),(0,4),(1,3),(1,4)).} 
& \makecell {
((0,1 ),(1,2 ),(3, 4))\;;\\
((3,4))\;;\\
((0,1),(1,2)).\\
}
\\
 
    \hline
        \small{
        $\left[ \begin{array}{c}
        0\; 0\; 0\; 1\; 1 \\
        1\; 1\; 1\; 2\; 2 \\
        2\; 2\; 2\; 0\; 0 \\ 
        3\; 3\; 3\; 3\; 3 \\
        4\; 4\; 4\; 4\; 4
        \end{array} \right] $} 
            &
            \makecell{
 ((0, 3),(1,3), (2,3), (3,4))\;;\\
((0,3),(1,3),(2, 3))\;;\\
((3,4)).} & \makecell{
            ((0,1),(1,2),(3,4))\;;\\
((3,4))\;;\\
((0,1),(1,2)).} 

   \\

    \hline
        \small{
        $\left[ \begin{array}{c}
        0\; 0\; 0\; 1\; 2 \\
        1\; 1\; 1\; 2\; 0 \\
        2\; 2\; 2\; 0\; 1 \\ 
        3\; 3\; 3\; 3\; 3 \\
        4\; 4\; 4\; 4\; 4
        \end{array} \right] $} 
            &  	 
            \makecell{
            ((0,3),(1,3),(2,3),(0,4),(1,4),(2,4))\;;\\
            ((0,3),(1,3),(2,3))\;;\\
            ((3,4)).} & \makecell{
            ((0,1),(1,2))\;;\\
            ((0,1)).} \\
            
    \hline
        \small{
        $\left[ \begin{array}{c}
        0\; 0\; 0\; 0\; 0\\
        1\; 1\; 1\; 1\; 1\\
        2\; 2\; 2\; 2\; 2\\
        4\; 4\; 4\; 3\; 3\\
        3\; 3\; 3\; 4\; 4
        \end{array} \right] $}
            & \makecell{\\
            ((0,1),(0,2))\;;\\
            ((0,1),(1,3),(1,4))\;;\\
            (0,1),(1,2),(2,3),(2,4)).\\
           \\
           \\
           }	&
             \makecell{\\
             ((0,1),(0,2))\;;\\
            ((0,1),(1,2),(3,4))\;;\\
            ((0,1),(3,4))\;;\\ 
            ((3,4)).\\
            }\\
    \hline
        \small{
        $\left[ \begin{array}{c}
        0\; 0\; 0\; 0\; 0\\
        1\; 1\; 1\; 2\; 2\\
        2\; 2\; 2\; 1\; 1\\
        4\; 4\; 4\; 3\; 3\\
        3\; 3\; 3\; 4\; 4
        \end{array} \right] $} 
            & \makecell{((0,1),(0,2))\;;\\
            ((0,1),(1,3),(1,4))\;;\\
            ((1,3),(1,4),(2,3),(2,4)).}
            & \makecell{\\
            ((0,1),(0,2))\;;\\
            ((0,1),(1,2),(3,4))\;;\\
            ((0,1),(3,4))\;;\\ 
            ((3,4)).\\
            \\
            }\\ 
    \hline
        \small{
        $\left[ \begin{array}{c}
        0\; 0\; 0\; 1\; 1\\
        1\; 1\; 1\; 2\; 2\\
        2\; 2\; 2\; 0\; 0\\
        4\; 4\; 4\; 3\; 3\\
        3\; 3\; 3\; 4\; 4
        \end{array} \right] $} 
            & \makecell{((0,4),(1,4),(2,4),(0,3),(1,3),(2,3)).}
            	& \makecell{\\
            	((0,1),(0,2))\;;\\
            ((0,1),(1,2),(3,4))\;;\\
            ((0,1),(3,4))\;;\\ 
            ((3,4)).\\
            \\
            } \\

    \hline
        \small{
        $\left[ \begin{array}{c}
        0\; 0\; 0\; 0\; 0\\
        1\; 1\; 4\; 2\; 3\\
        2\; 3\; 2\; 4\; 1\\
        3\; 4\; 1\; 3\; 2\\
        4\; 2\; 3\; 1\; 4
        \end{array} \right] $} 
            & \makecell{\\
            \\
            ((0,1),(0,2),(0,3),(0,4)).\\
            \\
            \\
            } & None \\ 
    \hline
        \small{
        $\left[ \begin{array}{c}
        0\; 0\; 1\; 1\; 1 \\
        1\; 1\; 0\; 0\; 0 \\
        2\; 2\; 2\; 2\; 3 \\
        3\; 3\; 3\; 3\; 2 \\
        4\; 4\; 4\; 4\; 4
        \end{array} \right] $} 
            & \makecell{\\
((1,2),(0,2),(1,3),(0,3),(2,4),(3,4))\;;\\
((1,4),(0,4))\;;\\
((1,2),(0,2),(1,3),(0,3)).
\\
\\
\\
} & \makecell{((0,1),(2,3))\;; \\
((0,1)).}\\

    \hline
\end{tabular}
\end{center}
\end{table} 

 \begin{table}[H]
\caption{Continuous posets on quandles of order 5, Part III}
\label{Table1}
\begin{center}
\begin{tabular}{ |c|c|c|} 
    \hline
       Quandles for  n = 5 & 
        Right continuous & 
        Left continuous 
         \\ 
    \hline
        \small{
        $\left[ \begin{array}{c}
       0\; 0\; 1\; 1\; 1\\
1\; 1\; 0\; 0\; 0\\
2\; 2\; 2\; 4\; 3\\
3\; 3\; 4\; 3\; 2\\
4\; 4\; 3\; 2\; 4
        \end{array} \right] $} 
            & \makecell{\\
            \\
             ((0,2),(0,3),(0,4),(1,2),(1,3),(1,4)).\\
                               \\
                               \\
                               \\}

            & \makecell{((0,1),(2,3))\;; \\
((0,1)).}
	 \\ 
    \hline
        \small{
        $\left[ \begin{array}{c}
       0\; 0\; 1\; 1\; 1\\
1\; 1\; 0\; 0\; 0\\
2\; 2\; 2\; 2\; 2\\
4\; 4\; 4\; 3\; 3\\
3\; 3\; 3\; 4\; 4
        \end{array} \right] $} 
            & \makecell{\\
            ((0,2),(0,3),(0,4),(1,2),(1,3),(1,4))\;;\\
            ((0,2),(1,2))\;;\\
            ((2,3),(2,4))\;;\\
            ((0,3),(0,4),(1,3),(1,4)).\\
            \\
            }
            & \makecell{((0,1),(2,3)) \;;\\
((0,1)).}\\
	  
    \hline
        \small{
        $\left[ \begin{array}{c}
       0\; 0\; 1\; 1\; 1\\
1\; 1\; 0\; 0\; 0\\
3\; 4\; 2\; 4\; 3\\
4\; 2\; 4\; 3\; 2\\
2\; 3\; 3\; 2\; 4
        \end{array} \right] $} 
            &  \makecell{\\
            \\
             ((0,2),(0,3),(0,4),(1,2),(1,3),(1,4)).\\
                               \\
                               \\
                               \\}	& None \\ 
    \hline
        
\end{tabular}

\end{center}
\end{table} 
 \section*{Acknowledgement} 
Mohamed Elhamdadi was partially supported by Simons Foundation collaboration grant 712462.

%%%%%%%%%%%%%%%%%%%%%%
%%%%%%%%%%%%%%%%%%%%%
%%%%%%%%%%%%%%%%%%%%%%%
%%%%%%%%%%%%%%%%%%%%%%%%
\end{document}